\documentclass{amsart}%
\usepackage{amsfonts}
\usepackage{amsmath}
\usepackage{amssymb}
\usepackage{graphicx}
\usepackage{enumerate}
\usepackage{xcolor}%
\usepackage[pdftex,pdfborderstyle={/S/U/W 1},hyperfootnotes=false,hidelinks]{hyperref}
\usepackage{mathtools}
\mathtoolsset{centercolon}
\newtheorem{theorem}{Theorem}
\theoremstyle{plain}

\newtheorem{corollary}[theorem]{Corollary}

\newtheorem{example}{Example}

\newtheorem{proposition}[theorem]{Proposition}
\newtheorem{remark}[theorem]{Remark}

\numberwithin{equation}{section}

\newcommand{\assign}{:=}
\newcommand{\tmmathbf}[1]{\ensuremath{\boldsymbol{#1}}}
\newcommand{\tmop}[1]{\ensuremath{\operatorname{#1}}}

\begin{document}
\title[The {\L}ojasiewicz exponent]{The {\L}ojasiewicz exponent in non-degenerate deformations of surface singularities}
\address[S.\ Brzostowski, T.\ Krasi\'{n}ski, and G.\ Oleksik]{Faculty~of~Mathematics and Computer~Science\\
University of {\L}{\'o}d{\'z}\\
ul. Banacha 22, 90-238 {\L}{\'o}d{\'z}\\
Poland}
\author{Szymon Brzostowski}
\email{szymon.brzostowski@wmii.uni.lodz.pl}
\author{Tadeusz Krasi\'{n}ski}
\email{tadeusz.krasinski@wmii.uni.lodz.pl}
\author{Grzegorz Oleksik}
\email{grzegorz.oleksik@wmii.uni.lodz.pl}

\begin{abstract}
We prove the constancy of the \L ojasiewicz exponent in non-degenerate $\mu
$-constant deformations of surface singularities. This is a positive answer to
a question posed by B.\ Teissier.

\end{abstract}
\maketitle

\section{Introduction}

Let $f(z)=f(z_{1},\ldots,z_{n})\in\mathbb{C}\{z_{1},\ldots,z_{n}%
\}=:\mathcal{O}_{n}$ be a convergent power series defining an \textit{isolated
singularity }at the origin $0\in\mathbb{C}^{n}$, i.e., $f(0)=0$ and the
gradient of $f$,%
\[
\nabla f:=\left(  \frac{\partial f}{\partial z_{1}},\ldots,\frac{\partial
f}{\partial z_{n}}\right)  :(\mathbb{C}^{n},0)\rightarrow(\mathbb{C}^{n},0),
\]
has an isolated zero at $0\in\mathbb{C}^{n}$. The \textit{\L ojasiewicz
exponent }$\mathcal{L}_{0}(f)$ of $f$ is the smallest ($=$ infimum) $\theta>0$
such that there exists a neighbourhood $U$ of $0\in\mathbb{C}^{n}$ and a
constant $C>0$ such that%
\[
\left\vert \nabla f(z)\right\vert \geq C\left\vert z\right\vert ^{\theta
}\text{ \ \ \ for }z\in U.  
\]

\begin{remark}
One can similarly define the \L ojasiewicz exponent $\mathcal{L}%
_{0}(F)$ of any holomorphic mapping $F:(\mathbb{C}^{n},0)\rightarrow
(\mathbb{C}^{p},0)$ having an isolated zero at $0\in\mathbb{C}^{n}$ (replacing
$\nabla f$ by $F$ in the above definition). Moreover, using the usual conventions for the infimum,
one can also extend this definiton to any holomorphic mapping $F:(\mathbb{C}^{n},0)\rightarrow
\mathbb{C}^{p}$; then $\mathcal{L}_{0}(F)=0$ if $F(0)\neq 0$, and $\mathcal{L}_{0}(F)=+\infty$
if $F(0)=0$ but the zero is not isolated.
\end{remark}

It is known $\mathcal{L}_{0}(f)$ is a rational positive number and it is an
analytic invariant of $f$. Moreover, it depends only on the ideal $(\partial f/\partial
z_{1},\ldots,\partial f/\partial z_{n})$ in $\mathcal{O}_{n}$ and can be calculated
using analytic paths, i.e.,
\[
\mathcal{L}_{0}(f)=\sup_{\Phi}\frac{\operatorname*{ord}\left(  \nabla
f\circ\Phi\right)  }{\operatorname*{ord}\Phi}=\max_{\Phi}\frac{\operatorname*{ord}\left(  \nabla
f\circ\Phi\right)  }{\operatorname*{ord}\Phi},
\]
where $0\neq\Phi=(\varphi_{1},\ldots,\varphi_{n})\in\mathbb{C}\{t\}^{n},$
$\Phi(0)=0$, and $\operatorname*{ord}\Phi:=\min_{i}\operatorname*{ord}%
\varphi_{i}$. It is also known that $[\mathcal{L}%
_{0}(f)]+1$ is the degree of $C^{0}$-sufficiency of $f$; nevertheless,
it is still an open and difficult problem whether the \L ojasiewicz
exponent is a topological invariant, i.e., whether $\mathcal{L}_{0}%
(f)=\mathcal{L}_{0}(g)$ if isolated singularities $f$ and $g$ are
topologically $\mathcal{R}$-equivalent (this is true for $n=2$). The behaviour of
$\mathcal{L}_{0}(f)$ in analytic families of singularities is also enigmatic.
As much as the Milnor number $\mu(f)=\dim_{\mathbb{C}} \mathcal{O}_n / ( \frac{\partial f}{\partial z_1},
\ldots, \frac{\partial f}{\partial z_n} )$ of $f$ is semi-continuous from above in
such families (and this is a quite natural property because the Milnor number is
defined as the multiplicity of a mapping -- exactly the gradient mapping),
the \L ojasiewicz exponent has no such property. Using, for instance, the
formula for the \L ojasiewicz exponent for functions of two variables given by
A.\ Lenarcik \cite{Len98} (see Remark \ref{lenny}), we easily verify

\begin{example}
1. For the family $f_{t}(x,y)=x^{2}+ty^{2}+y^{3},$ $t\in\mathbb{C},$ we have
$\mathcal{L}_{0}(f_{0})=2$ and $\mathcal{L}_{0}(f_{t})=1$ for $t\neq0.$

2. For the family $f_{t}(x,y)=xy^{5}+tx^{2}+x^{8},$ $t\in\mathbb{C},$ we have
$\mathcal{L}_{0}(f_{0})=7$ and $\mathcal{L}_{0}(f_{t})=9$ for $t\neq0.$
\end{example}

This shows that in order to obtain positive results on the \L ojasiewicz exponent in families of
singularities, one has to impose some assumptions on the families. B.\ Teissier
in \cite{Tei77a} proved that, under the additional assumption of $\mu$-constancy of the
family, the \L ojasiewicz exponent is semi-continuous from below. This
result has been generalized to mappings by A.\ P\l oski in \cite{Plo10}. In the
aforementioned paper B.\ Teissier posed the problem whether in this case (i.e.\ for a $\mu
$-constant family) $\mathcal{L}_{0}(f_{t})$ must also be constant. For $n=2$,
i.e.,\ in the case of plane curve singularities, it is true because $\mu
$-constancy in a family of plane curve singularities $(f_{t}(x,y))$ implies
the curves $\{f_{t}(x,y)=0\}$ are pairwise topologically equivalent 
(by the L\^{e}-Ramanujam theorem, for instance). This, in turn, gives the constancy
of $\mathcal{L}_{0}(f_{t})$ as the \L ojasiewicz exponent is a topological
invariant of such singularities (see \cite[Cor.\ 1.5]{Plo01} for a formula for the
\L ojasiewicz exponent expressed in terms of the Puiseux characteristics of the branches of
a curve as well as their intersection multiplicities).

In the article, we solve Teissier's problem in the affirmative for surface singularities, in the
particular case of non-degenerate deformations. More precisely, we prove that in $\mu
$-constant non-degenerate families of surface singularities, the \L ojasiewicz
exponent is also constant. In a similar spirit, we remark that the constancy of the multiplicity ($=$ the order of
$f_{t}$ at $0$) in $\mu$-constant non-degenerate families of hypersurface
singularities of any dimension, which is a particular case of the famous Zariski problem,
has already been settled by Y.\ O.\ M.\ Abderrahmane in \cite{Abd16}.

The proof of our result is based on \cite{BKO20}, where we gave an
explicit formula for the \L ojasiewicz exponent of a non-degenerate surface
singularity in terms of its Newton polyhedron (see formula \eqref{wzor_bko} below),
and on \cite{BKW19} which gives a characterization of $\mu$-constant non-degenerate families of
surface singularities. 
We notice that the characterization from \cite{BKW19} 
has recently been extended (in an equivalent form) to any dimension by M.\ Leyton-{\'A}lvarez,
H.\ Mourtada and M.\ Spivakovsky \cite[Thm.\ 2.15]{LMS20}. Among other papers concerning the properties of
$\mu$-constant non-degenerate families of hypersurface
singularities, one can mention \cite{Oka18}, \cite{Abd16}.

\section{The Newton polyhedron of a singularity}

Let $0\neq f:(\mathbb{C}^{n},0)\rightarrow(\mathbb{C},0)$ be a holomorphic
function defined by a convergent power series $\sum_{\nu\in\mathbb{N}^{n}%
}a_{\nu}z^{\nu},$ $z=(z_{1},\ldots,z_{n}).$ Let $\mathbb{R}_{+}^{n}%
:=\{(x_{1},\ldots,x_{n})\in\mathbb{R}^{n}:x_{i}\geq0,\ i=1,\ldots,n\}.$ We
define $\operatorname*{supp}f:=\{\nu\in\mathbb{N}^{n}:a_{\nu}\neq
0\}\subset\mathbb{R}_{+}^{n}.$ In the sequel, we will identify $\nu=(\nu
_{1},\ldots,\nu_{n})\in\operatorname*{supp}f$ with their associated monomials
$z^{\nu}=z_{1}^{\nu_{1}}\cdots z_{n}^{\nu_{n}}.$ We define the \textit{Newton
polyhedron} $\Gamma_{+}(f)\subset\mathbb{R}_{+}^{n}$ \textit{of} $f$ as the
convex hull of $\{\nu+\mathbb{R}_{+}^{n}:\nu\in\operatorname*{supp}f\}.$ We
say $f$ is \textit{convenient} if $\Gamma_{+}(f)$ has non-empty intersection
with each coordinate $x_{i}$-axis $(i=1,\ldots,n)$. Let $\Gamma(f)$ be the
set of compact boundary faces of any dimension of $\Gamma_{+}(f)$ -- the
\textit{Newton boundary }of $f$. Denote by $\Gamma^{k}(f)$ the set of all
$k$-dimensional faces of $\Gamma(f),$ $k=0,\ldots,n-1.$ Then $\Gamma(f)=%
{\bigcup_{k}}\Gamma^{k}(f)$. For each face $S\in\Gamma(f)$, we define the quasihomogeneous
polynomial $f_{S}:=\sum_{\nu\in S}a_{\nu}z^{\nu}$. We say $f$ is
\textit{non-degenerate on} $S$ if the system of polynomial equations
$\{\frac{\partial f_{S}}{\partial z_{1}}= \dots =\frac{\partial f_{S}}{\partial z_{n}}=0\}$ has no solution in
$(\mathbb{C}^{\ast})^{n}$; $f$ is \textit{non-degenerate (in the Kushnirenko
sense)}\/ if $f$ is non-degenerate on each face $S\in\Gamma(f)$.

For each $(n-1)$-dimensional (compact) face $S\in\Gamma^{n-1}(f)$ the unique
affine hyperplane $\Pi_{S}$ containing $S$ intersects each coordinate
$x_{i}$-axis in a point with a positive $x_{i}$-coordinate $m(S)_{x_{i}}.$ We
define%
\[
m(S):=\max_{i}\, m(S)_{x_{i}}.
\]

It is interesting that for typical non-degenerate isolated singularities $f$ in
$\mathbb{C}^{2}$ and $\mathbb{C}^{3}$ the \L ojasiewicz exponent $\mathcal{L}_{0}(f)$ can be
expressed in terms of $m(S)$ where $S$ run over some special $(n-1)$%
-dimensional faces of $\Gamma(f).$ We define them now.

We say $S\in\Gamma^{n-1}(f)$ is \textit{exceptional with respect to the
}$x_{i}$-\textit{axis} if one of the partial derivatives $\frac{\partial
f_{S}}{\partial z_{j}},$ $j\neq i,$ is a pure power of $z_{i}.$ Geometrically,
this means $S$ is an $(n-1)$-dimensional pyramid with the base lying in one of
the $(n-1)$-dimensional coordinate hyperplanes containing the $x_{i}$-axis and
with the apex lying at distance $1$ from this axis (see Figure \ref{excfig}).
We denote the set of exceptional faces of $f$ with respect to $x_{i}$-axis by
$E_{x_{i}}(f).$%

\begin{figure}
\includegraphics[
width=4.4843in
]%
{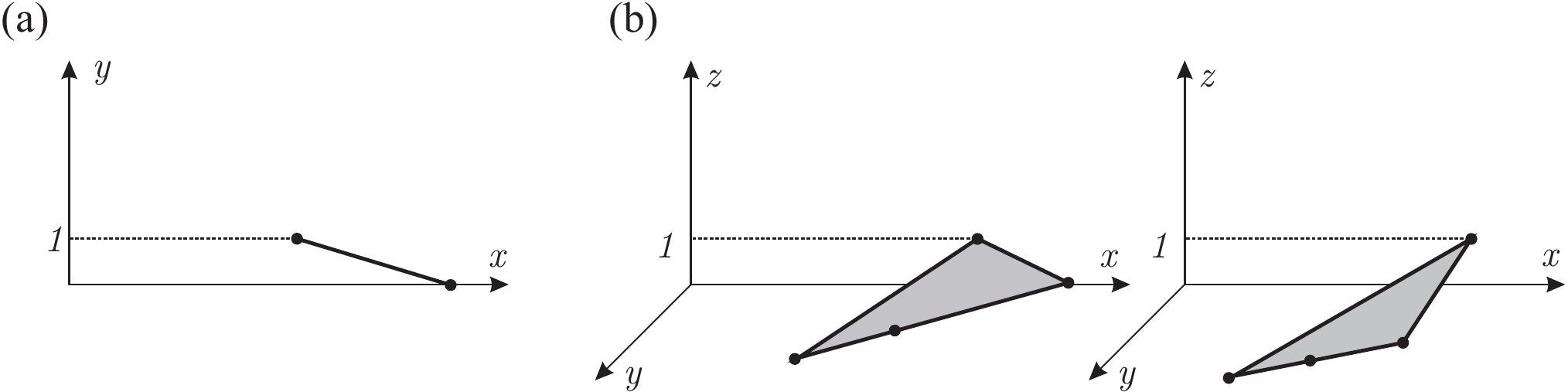}%
\caption{Exceptional faces with respect to the $x$-axis in two and three dimensions}\label{excfig}
\end{figure}

A face $S\in\Gamma^{n-1}(f)$ is \textit{exceptional }if $S$ is exceptional
with respect to some axis. We denote the set of all exceptional faces of $f$ by
$E(f)$. Then the formula for the \L ojasiewicz exponent given in \cite{BKO20} reads

\begin{theorem}\label{bko}
If $f:(\mathbb{C}^{3},0)\rightarrow(\mathbb{C},0)$ is a
non-degenerate isolated surface singularity possessing non-exceptional faces,
i.e., $\Gamma^{2}(f)\setminus E(f)\neq\emptyset$, then%
\begin{equation}\label{wzor_bko}
\mathcal{L}_{0}(f)=\max_{S\in\Gamma^{2}(f)\setminus E(f)}m(S)-1.
\end{equation}

\end{theorem}

\begin{remark}
The case\/ $\Gamma^{2}(f)\setminus E(f)=\emptyset$ is relatively simpler. The
third-named author in \cite[Prop.\ 3.4, Thm.\ 3.8]{Ole13} showed that in this case, if we denote the
variables in $\mathbb{C}^{3}$ by $x,y,z$, there is exactly one segment
$S\in\Gamma^{1}(f)$ joining monomials $xy$ and $z^{k},k\geq2$ (up to permutation
of the variables), and then $\mathcal{L}_{0}(f)=k-1$.
\end{remark}

\begin{remark}\label{lenny}
A.\ Lenarcik in \cite{Len98} proved an alike formula for $n=2$. Precisely,%
\[
\mathcal{L}_{0}(f)=\left\{
\begin{tabular}
[c]{lll}%
$\max_{S\in\Gamma^{1}(f)\setminus E(f)}m(S)-1,$ & if & $\Gamma^{1}(f)\setminus
E(f)\neq\emptyset$\\
$1,$ & if & $\Gamma^{1}(f)\setminus E(f)=\emptyset$%
\end{tabular}
\right.\!\! .
\]

\end{remark}

\begin{remark}\label{pytanko}
It is an open problem if a formula of the above type holds also in the $n$-dimensional case $(n>3)$.
\end{remark}

It turns out the set of non-exceptional faces in formula \eqref{wzor_bko} could be
narrowed -- it suffices to allow only non-exceptional faces having ``furthest
intersections'' with the axes. The definition capturing such faces for surface singularities is as follows.

We say $S\in\Gamma^{2}(f)$ is \textit{proximate for the }$x_{i}%
$\textit{-axis }$(i \in \{1,2,3\})$ if $S$ is a non-exceptional face with respect to the $x_{i}%
$-axis ($S\notin E_{x_{i}}(f))$, has a vertex either on $x_{i}$-axis or
lying at distance $1$ from this axis, and
touches both coordinate planes containing this axis.
Possible proximity faces are illustrated in Figure \ref{prf}.
\begin{figure}
\includegraphics[
height=1.062in,
width=3.4843in
]%
{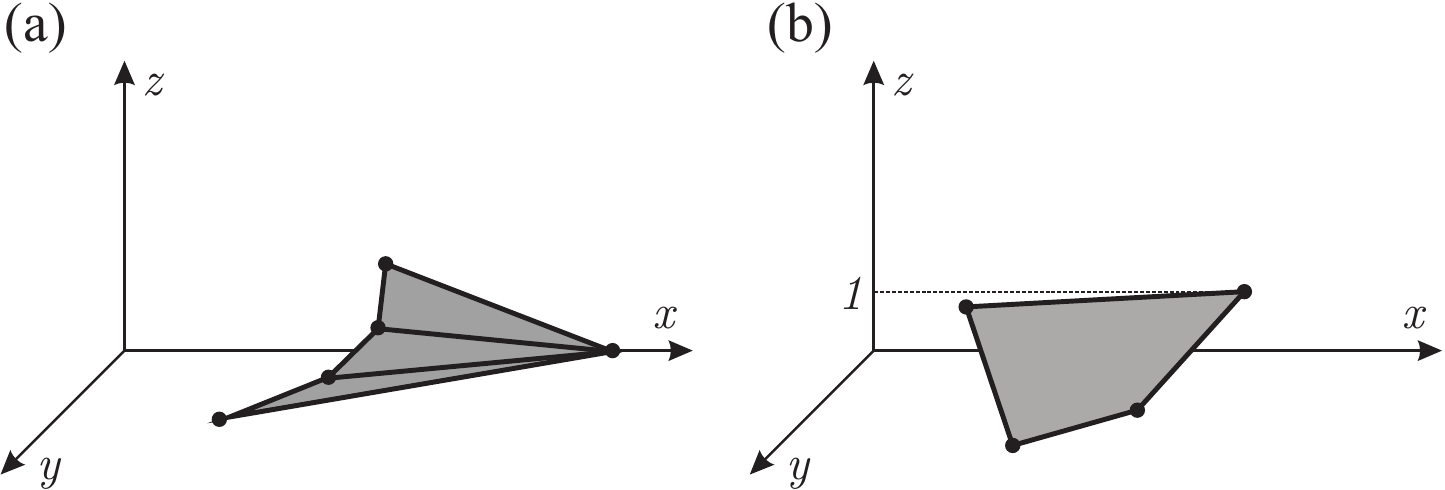}%
\caption{Proximity faces for the $x$-axis: (a) convenient case, (b) non-convenient case.}\label{prf}
\end{figure}

It is easy to prove (cf.\ Lemma 3.1 and Theorem 3.8 in \cite{Ole13} and Lemma 6 in \cite{BKO20}) the
following properties of proximity faces.

\begin{proposition}
\label{prox}
{\hfill}

1. If\/ $\Gamma^{2}(f)\setminus E(f)\neq\emptyset$, then each axis
has a proximate face; moreover, all these faces are non-exceptional.

2. Once exists, a proximate face for a given axis is unique (see Figure  \ref{prf}(b)) if it is
non-convenient with respect to this axis (and not necessarily unique
in the opposite case; see Figure  \ref{prf}(a)).

3. Let $S$ be a proximate face for the $x$-axis.
Then the supporting plane of $S$ has the highest coordinate of intersection
with the $x$-axis among all the $x$-non-exceptional faces.
\end{proposition}

The assumption $\Gamma^{2}(f)\setminus E(f)\neq\emptyset$ in item 1.\ of the above
proposition is essential.

\begin{example}
The surface singularity $f(x,y,z):=xz+yz+y^{3}\in\mathcal{O}^{3}$ has a unique
$2$-dimensional face. The face is exceptional, $x$-proximate, $y$-proximate and
not $z$-proximate.
\end{example}

The formula we will use in the proof of our main theorem is a stronger version
of Theorem \ref{bko}.

\begin{theorem}[{\cite[Cor.\ 2]{BKO20}}]
\label{thpro} If $f:(\mathbb{C}^{3},0)\rightarrow
(\mathbb{C},0)$ is a non-degenerate isolated surface singularity possessing
non-exceptional faces, i.e., $\Gamma^{2}(f)\setminus E(f)\neq\emptyset$, and
$S_{x_{i}}$ is any proximate face for the $x_i$-axis $(i=1,2,3)$, then%
\begin{equation}
\mathcal{L}_{0}(f)=\max_{i}\, m(S_{x_{i}})_{x_{i}}-1.
\end{equation}

\end{theorem}

\section{The main Result}

First, we recall some notions. Let $f_{0}\in\mathcal{O}_{n}$ be an isolated
singularity.
If $f_{0}$ is convenient, then $\nu(f_{0})$ is the \textit{Newton
number of $f_{0}$} (\cite{Kou76}). A \textit{deformation of $f_{0}$} is a holomorphic
function germ $f(t,z):(\mathbb{C\times C}^{n},0)\rightarrow(\mathbb{C},0)$ such
that $f(0,z)=f_{0}(z)$ and $f(t,0)=0.$ Each deformation $f(t,z)$ will also be
treated as a family $(f_{t})$ of germs at $0\in\mathbb{C}^{n}$ by putting
$f_{t}(z):=f(t,z)$. Since $f_{0}$ has an isolated critical point at $0$,
$f_{t}$ also has isolated critical points near the origin for sufficiently
small $t$ (\cite{GLS07}, Thm.\ 2.6 in Ch.\ I). In particular, $f_{t}$ has an
isolated critical point or a regular point $(\nabla f_{t}(0)\neq0)$ at $0$.
A deformation $(f_{t})$ of $f_{0}$ is \textit{non-degenerate} if $f_{t}$ are non-degenerate
for all small $t$. In particular, $f_{0}$ is a non-degenerate isolated singularity.
The main result of the paper is

\begin{theorem}
\label{lconstant} Let $f_{0}:(\mathbb{C}^{3},0)\rightarrow(\mathbb{C},0)$ be
a non-degenerate isolated surface singularity and let $(f_{t})$ be its
non-degenerate $\mu$-constant deformation. Then $\mathcal{L}_{0}%
(f_{t})=\mathcal{L}_{0}(f_{0})$ for small $t.$
\end{theorem}

To prove this theorem, we give some auxiliary facts. The first one is a result by
A.\ P\l oski \cite[Cor.\ 1.4]{Plo84}.

\begin{proposition}\label{rzad}
Let $f:(\mathbb{C}^{n},0)\rightarrow(\mathbb{C},0)$ be an
isolated singularity. Then
\[
\operatorname{rank}\left[  \frac{\partial^{2}f}{\partial z_{i}\partial z_{j}%
}(0)\right]  \geq n-1
\]
if and only if $\mu(f)=\mathcal{L}_{0}(f).$
\end{proposition}

The second fact is the following observation.
\begin{proposition}\label{vertex}
  Let $f_0 : (\mathbb{C}^3, 0) \rightarrow (\mathbb{C}, 0)$ be an isolated
  surface singularity of order $2$. If $(f_t)$ is a $\mu$-constant deformation
  of $f_0$, then $\mathcal{L}_0 (f_t) =\mathcal{L}_0 (f_0)$ for small $t$.
\end{proposition}

\begin{proof}
  Let $\tmmathbf{H} (f_t)$ denote the Hessian matrix of $f_t$ at $0$.
  By the assumption, we have $\tmop{rank}\tmmathbf{H} (f_0) > 0$.
  If $\tmop{rank}
  \tmmathbf{H} (f_0) \geqslant 2$, then $\tmop{rank} \tmmathbf{H} (f_t)
  \geqslant 2$ for small $t$ so by Proposition \ref{rzad} we get $\mathcal{L}_0 (f_t) =
  \mu (f_t) = \mu (f_0) =\mathcal{L}_0 (f_0)$. Let, now, $\tmop{rank}
  \tmmathbf{H} (f_0) = 1$, and let $x$, $y$, $z$ denote the variables in
  $\mathbb{C}^3$. Using splitting lemma (see \cite[Thm.\ 2.47]{GLS07}), after a
  holomorphic change of coordinates,
  we may assume that $f_0 = x^2 + g_0 (y, z)$, where $\tmop{ord} g_0 \geqslant
  3$. Hence, $f_t = f_0 + th_t (x, y, z)$, where $\tmop{ord} h_t \geqslant 2$.
  Applying the procedure from the splitting lemma to the variable $x$, we easily
  find that there exists a holomorphic change of coordinates of the form
  $x \mapsto \Phi (t, x, y, z)$, $y \mapsto y$, $z \mapsto z$, $t \mapsto t$, where $\tmop{ord} \Phi (0, x, 0, 0) = 1$,
  bringing $f_t$ into the form
  $f_t = f_0 + t \tilde{h}_t (y, z)$, where
  $\tmop{ord} \tilde{h}_t \geqslant 2$ for small $t$. Since both $\mu$ and
  $\mathcal{L}_0$ are invariants of the stable equivalence (see, e.g.,\ \cite[Thm.\ 21]{Oka18}), we may remove
  $x^2$ from $f_t$ and infer that $(g_t) \assign (g_0 + t \tilde{h}_t (y, z))$ is a
  $\mu$-constant deformation of the isolated plane curve singularity $g_0$ and
  $\mathcal{L}_0 (g_t) =\mathcal{L}_0 (f_t)$. As explained in the
  Introduction, this implies that $\mathcal{L}_0 (g_t) =\mathcal{L}_0 (g_0)$
  for small $t$, hence also $\mathcal{L}_0 (f_t) =\mathcal{L}_0 (f_0)$.
\end{proof}

As a corollary, we prove a particular case of the main theorem (in fact, an even stronger
assertion since we do not assume non-degeneracy of $f_{t}$) when the set of
$2$-dimensional faces $\Gamma^{2}(f_{0})$ is empty or consists of exceptional faces only.
\begin{corollary}
\label{except} Let $f_{0}:(\mathbb{C}^{3},0)\rightarrow(\mathbb{C},0)$ be an
isolated surface singularity such that $\Gamma^{2}(f_{0})\setminus
E(f_{0})=\emptyset.$ If $(f_{t})$ is a $\mu$-constant deformation of $f_{0},$
then $\mathcal{L}_{0}(f_{t})=\mathcal{L}_{0}(f_{0})$ for small $t.$
\end{corollary}
\begin{proof}
The assumption
$\Gamma^{2}(f_{0})\setminus E(f_{0})=\emptyset$ implies by Theorem 1.8 in
\cite{Ole13} that $\Gamma^{1}(f_{0})$ has the unique edge joining vertices
$z^{k}$ and $xy$, $k\geq2$ (up to permutation of variables $x,y,z$). Since
$xy$ is a vertex of $\Gamma_{+}(f_{0})$, Proposition \ref{vertex} delivers the assertion.
\end{proof}

Now we may prove the main theorem.

\begin{proof}
[Proof of the main theorem.]Let $f_{0}:(\mathbb{C}^{3},0)\rightarrow
(\mathbb{C},0)$ be a non-degenerate isolated singularity and let $(f_{t})$ be
its non-degenerate $\mu$-constant deformation. If $\Gamma^{2}(f)\setminus
E(f)=\emptyset$, then the assertion follows from Corollary \ref{except}. Assume
now $\Gamma^{2}(f)\setminus E(f)\neq\emptyset$. Then by Proposition
\ref{prox}, item 1 each coordinate axis has a proximate face and these faces are all
non-exceptional. Observe that it is enough to prove the theorem for
convenient $f_{0}$ and $f_{t}.$ In fact, if $f_{0}$ is non-convenient with
respect to the $x_{i}$-axis, $i\in\{1,2,3\},$ then we put $\widetilde{f_{0}%
}:=f_{0}+z_{i}^{k}$, $\widetilde{f_{t}}:=f_{t}+z_{i}^{k}$, where $k>\mu
(f_{0})=\mu(f_{t})$.
Possibly increasing $k$, we get that $\widetilde{f_{t}}$ are also non-degenerate (see for
instance \cite[Lemma 3.7]{BO16}), $\mu(\widetilde{f_{0}})=\mu(f_{0})=\mu
(f_{t})=\mu(\widetilde{f_{t}})$ and, by the P\l oski theorem \cite{Plo84},
$\mathcal{L}_{0}(f_{t})=\mathcal{L}_{0}(\widetilde{f_{t}})$. Hence, in the
sequel we may assume $f_{0}$ and $f_{t}$ are convenient. As $f_{t}$ are
non-degenerate, by the Kushnirenko theorem we have $\nu(f_{t})=\mu
(f_{t})=\operatorname*{const}$. By the monotonicity of the Newton number with
respect to Newton polyhedra (see e.g.\ \cite{Gwo08}, \cite[Cor.\ 2.3]{LMS20}), it is enough
to prove the theorem for monomial deformations, i.e., deformations of the form $f_{t}=f_{0}+tz^{\alpha},$
where $\alpha\notin\Gamma_{+}(f_{0})$ (the case $\alpha\in\Gamma_{+}(f_{0})$ follows from
Theorem \ref{bko}). We may also assume $\Gamma_{+}%
(f_{t})=\operatorname*{const}$ for sufficiently small $t\neq0$.

Denote the variables in $\mathbb{C}^{3}$ by $x,y,z.$ Then $f_{t}%
(x,y,z)=f_{0}(x,y,z)+tx^{\alpha_{1}}y^{\alpha_{2}}z^{\alpha_{3}},$
$(\alpha_{1},\alpha_{2},\alpha_{3})\notin\Gamma_{+}(f_{0})$. Since $\nu
(f_{t})$ is constant and $\Gamma_{+}(f_{0})\subset\Gamma_{+}%
(f_{t})$, using {\cite[Theorem 1]{BKW19}} we infer that the vertex $C:=x^{\alpha_{1}%
}y^{\alpha_{2}}z^{\alpha_{3}}$ of $\Gamma_{+}(f_{t})$ lies in one of the
coordinate planes $H$, say the $xy$-plane, and $P:=\overline{\Gamma_{+}(f_{t}%
)\setminus\Gamma_{+}(f_{0})}$ is a pyramid with the base $\overline{\Gamma
_{+}(f_{t})\setminus\Gamma_{+}(f_{0})}\cap H$ whose apex $D\in\Gamma
^{0}(f_{0})$ has its $z$-coordinate equal to one (see Figure \ref{pomoc}).
Since $\operatorname*{ord} f_0\geqslant 2$,
$D$ does not belong to the $z$-axis.%
\begin{figure}
\includegraphics[
width=4.2234in
]%
{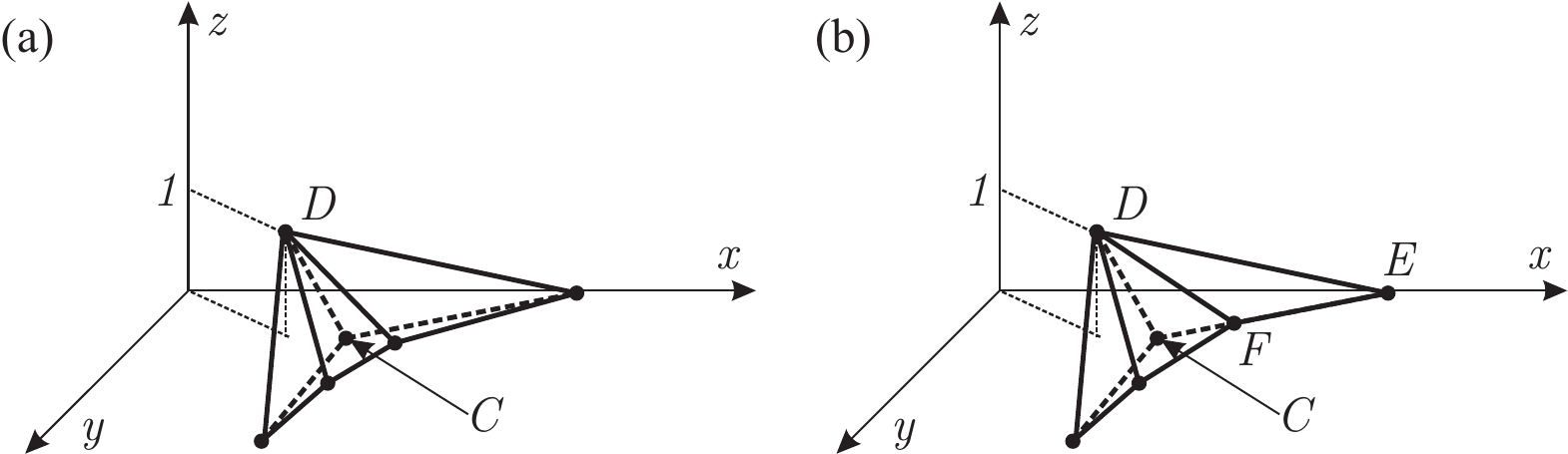}
\caption{The difference $\overline{\Gamma_+(f_t)\setminus\Gamma_+(f_0)}$}\label{pomoc}
\end{figure}

Observe that $\Gamma^{2}(f_{t}),$ $t\neq0,$ differ from $\Gamma^{2}(f_{0})$ on some
$2$-dimensional polytopes (precisely triangles).
Write $\Gamma^{2}(f_{t})=(\Gamma^{2}(f_0)\setminus
\{S_{1},\ldots,S_{l}\})\cup\{N_{1},\ldots,N_{k}\}$,
where $S_{1},\ldots,S_{l}$, $l\geq1$, are the sides that get
removed from $\Gamma^{2}(f_{0})$ and $N_{1},\ldots,N_{k}$, $k\geq1$, the new
ones showing in $\Gamma^{2}(f_{t})$. (As an illustration, in Figure \ref{pomoc}(a)
the three sides of $\Gamma^{2}(f_{0})$ get replaced by two new ones, marked with dashed contours,
taken from $\Gamma^{2}(f_{t})$;
in Figure \ref{pomoc}(b) the three sides of $\Gamma^{2}(f_{0})$ are also replaced by two new ones
but in this case one of them, namely $CDE$, is an extension of the one removed, i.e., $DEF$.
This also shows that $N_{i}$, $S_{j}$ need not constitute sides of $P$.)

Since $N_{i}$, $S_{j}$ are triangles with the
bases contained in the $xy$-plane and sharing the common vertex $D$ whose $z$-coordinate is equal to
$1$, we notice that if any of the faces $N_i$, $S_j$ is proximate for the $z$-axis, then 
either $D = (1,0,1)$ or $D = (0,1,1)$
(in the non-convenient case this is because any proximity face of an axis has a vertex lying at distance
$1$ from this axis).
But if this happens, we get the assertion by applying Proposition \ref{vertex}.

Thus, we may suppose none of the faces $N_i$, $S_j$ is proximate for the $z$-axis.
This also means that, in order to finish the proof, it is enough to show that 
the intersections of (the supporting planes of) the $x$- and $y$-proximate faces of $\Gamma_{+}(f_{t})$
with the corresponding axes are the same as those of $\Gamma_{+}(f_{0})$; indeed, the formula in Theorem
\ref{thpro} then implies that $\mathcal{L}_{0}(f_{t})=\mathcal{L}_{0}(f_{0})$ for small $t$.
To this end, we consider the following cases:

\begin{enumerate}[a)]
\item The vertex $D\not \in (xz\text{-plane}\cup yz\text{-plane})$. Then $C$ does not belong to the $x$-axis
and to the $y$-axis. \nopagebreak

\begin{figure}[!h]
\includegraphics[%
width=4.2234in
]%
{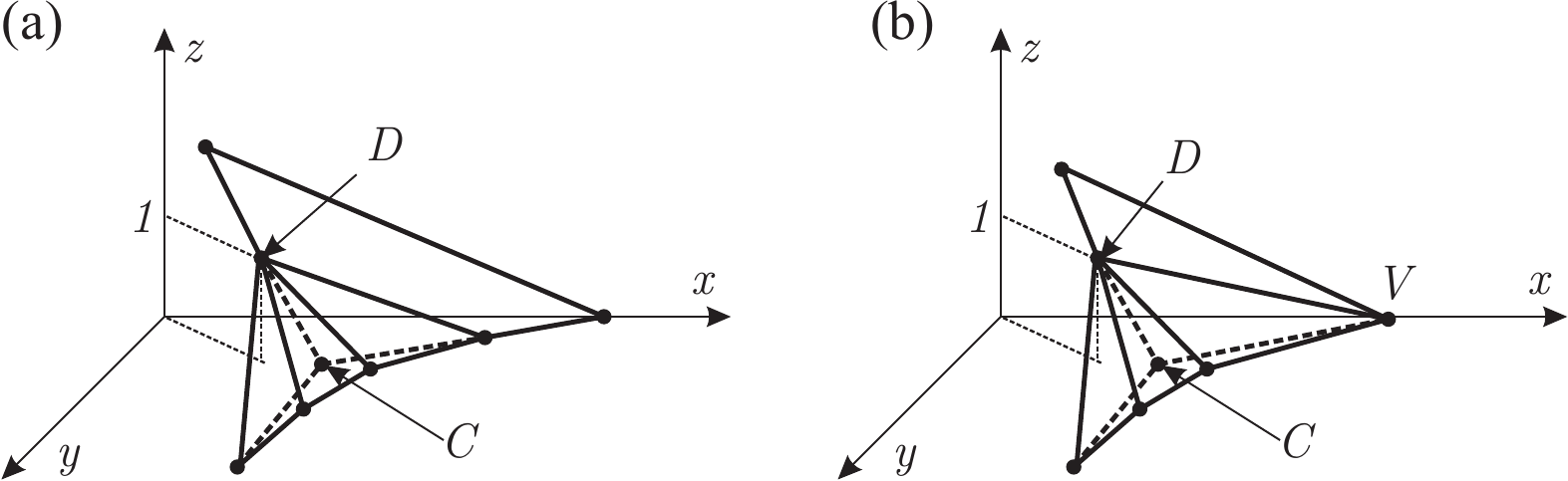}%
\caption{}\label{polozenie}
\end{figure}

If none of the faces $N_{1},\ldots,N_{k}$, $S_{1},\ldots,S_{l}$ touches the $x$- or $y$-axis,
then none of these faces touches the $xz$- or $yz$-plane, see Figure \ref{polozenie}(a). Hence none of these
faces is proximate for the $x$- or $y$-axis. Then the Newton polyhedron
$\Gamma_{+}(f_{t})$ of $f_{t}$, for $t\neq0$, has the same proximity faces, for all three axes, as
those of $\Gamma_{+}(f_{0})$. Consequently, also the intersections of (the supporting planes of) 
the proximity faces of $\Gamma_{+}(f_{t})$
with the axes are the same as those of $\Gamma_{+}(f_{0})$.

If some $S_{i}$ or $N_j$ touches the $x$-axis (for the $y$-axis we could reason analogously) in the vertex $V$ 
(see Figure \ref{polozenie}(b)), then $DV$ is the common edge of some $x$-proximate face  of $\Gamma_+(f_0)$ and some
 $x$-proximate face of $\Gamma_+(f_t)$. 
This means that the proximity faces for the $x$-axis of both $\Gamma_{+}(f_{0})$ and
$\Gamma_{+}(f_{t})$ intersect this axis in the same coordinate, viz., $V_x$.


\item The vertex $D\in (xz\text{-plane}\,\cup\,yz\text{-plane})$.  Without loss of generality, we may assume that $D$ belongs to the $xz$-plane
(see Figure \ref{touchxz}).
\begin{figure}[!h]
\includegraphics[%
width=2.1234in
]%
{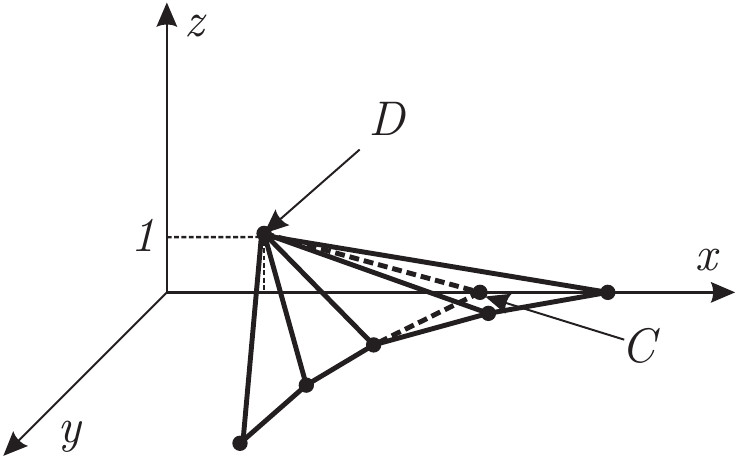}%
\caption{}\label{touchxz}
\end{figure}
Since its $z$-coordinate equals one, we infer that all the
faces $N_{1},\ldots,N_{k}$, $S_{1},\ldots,S_{l}$ are exceptional with respect to the $x$-axis. Combining this observation
with the assumption $\Gamma^2(f) \setminus E(f)\neq \emptyset$, we conclude that
$\Gamma_{+}(f_{t})$ and $\Gamma_{+}(f_{0})$ have the same proximity faces, for all the three axes, and so
$\mathcal{L}_{0}(f_{t})=\mathcal{L}_{0}(f_{0})$ for small $t$. 
\end{enumerate}
This ends the proof.
\end{proof}

\section{Concluding remarks}

Using the main theorem and some known facts, we may obtain additional information on
families of surface singularities.

\begin{corollary}
Let $f_{0}:(\mathbb{C}^{3},0)\rightarrow(\mathbb{C},0)$ be a non-degenerate
isolated surface singularity and let $(f_{t})$ be its non-degenerate
deformation. If $f_{t}$ are pairwise topologically equivalent, then
$\mathcal{L}_{0}(f_{t})=\mathcal{L}_{0}(f_{0})$ for small $t.$
\end{corollary}

\begin{proof}
Since $f_{t}$ are pairwise topologically equivalent, they have the same Milnor
numbers. So, $(f_{t})$ is $\mu$-constant deformation of $f_{0}.$ By the main
theorem, we get $\mathcal{L}_{0}(f_{t})=\mathcal{L}_{0}(f_{0})$ for small $t.$
\end{proof}

We predict the main theorem is true also in the $n$-dimensional case. In the proof
of our theorem we use two facts on surface singularities: the first one -- a
formula for the \L ojasiewicz exponent of non-degenerate surface singularities
(Theorem \ref{thpro}), and the second one -- the characterization of these Newton
polyhedra of surface singularities that have the same Newton numbers
{\cite[Thm.\ 1]{BKW19}}. The second fact has been
recently generalized by M.\ Leyton-{\'A}lvarez, H.\ Mourtada and M.\ Spivakovsky in
\cite[Thm.\ 2.15]{LMS20} to the $n$-dimensional case.  Since the
\L ojasiewicz exponent is one and the same for isolated non-degenerate singularities with a given Newton polyhedron
(\cite{Brz19}), there ``only'' remains the question about a formula for the exponent,
a formula expressed in terms of the Newton polyhedron (cf.\ Remark \ref{pytanko}).
\medskip

\bibliographystyle{plain}
\bibliography{LojasBKO11_arxiv}

\end{document}